\documentclass[sn-mathphys]{sn-jnl}

\usepackage{subcaption}

\jyear{2021}%

\theoremstyle{thmstyleone}%
\newtheorem{theorem}{Theorem}
\newtheorem{proposition}[theorem]{Proposition}%

\theoremstyle{thmstyletwo}%
\newtheorem{example}{Example}%
\newtheorem{remark}{Remark}%

\theoremstyle{thmstylethree}%
\newtheorem{definition}{Definition}%

\begin{document}

\title[Interval Extropy and Weighted Interval Extropy]{Interval Extropy and Weighted Interval Extropy}


\author*[1]{\fnm{Francesco} \sur{Buono}}\email{francesco.buono3@unina.it}
\equalcont{These authors contributed equally to this work.}

\author[2]{\fnm{Osman} \sur{Kamari}}\email{osman.kamari@uhd.edu.iq}
\equalcont{These authors contributed equally to this work.}

\author[3]{\fnm{Maria} \sur{Longobardi}}\email{maria.longobardi@unina.it}
\equalcont{These authors contributed equally to this work.}

\affil[1]{\orgdiv{Dipartimento di Matematica e Applicazioni ``Renato Caccioppoli"}, \orgname{Università degli Studi di Napoli Federico II}, \orgaddress{\city{Naples}, \state{Italy}}}

\affil[2]{\orgdiv{Department of Business Management}, \orgname{University of Human Development}, \orgaddress{ \city{Sulaymaniyah},  \state{Iraq}}}

\affil[3]{\orgdiv{Dipartimento di Biologia}, \orgname{Università degli Studi di Napoli Federico II}, \orgaddress{\city{Naples}, \state{Italy}}}

\abstract{Recently, Extropy was introduced by Lad, Sanfilippo and Agrò as a complement dual of Shannon Entropy. In this paper, we propose dynamic versions of Extropy for doubly truncated random variables as measures of uncertainty called Interval Extropy and Weighted Interval Extropy. Some characterizations of random variables related to these new measures are given. Several examples are shown. These measures are evaluated under the effect of linear transformations and, finally, some bounds for them are presented.}




\maketitle

\section{Introduction}
The concept of Shannon entropy as a basic measure of uncertainty for a random variable was introduced by Shannon \cite{Shannon}. Let $X$ be an absolutely continuous non-negative random variable having probability density function (pdf) $f$ and cumulative distribution function (cdf) $F$. In Reliability Theory, $X$ represents the lifetime of a system or a living organism. Shannon entropy for this kind of random variables is named differential entropy and is defined as follows: 
\begin{eqnarray}
\nonumber
H(X)=-\int_{0}^{+\infty}{f(x)}\log f(x)\,dx,
\end{eqnarray}
where $\log$ denotes the natural logarithm.
Recently, another measure of uncertainty, known as extropy, was proposed by Lad et al. \cite{Lad} as a dual measure of Shannon entropy. For a non-negative random variable $X$ the extropy is defined as below:
\begin{eqnarray}
J(X)=-\frac{1}{2} \int_{0}^{+\infty}{f^2(x)}\,dx. \label{extropy}
\end{eqnarray}
The concept of extropy is useful in many fields: for instance, it is applied in automatic speech recognition \cite{becerra}. In particular, the extropy of a network output with respect to the training set can be obtained in order to compute a kind of transformed cross entropy. Moreover, extropy is a measure better than entropy in some scenarios in statistical mechanics and thermodynamics \cite{martinas}. More recently, some applications of extropy have been done in pattern recognition \cite{balakrishnan2,kazemi}.

Qiu et al. \cite{Qiu} defined the extropy for residual lifetime $X_t=(X-t \mid X>t)$ whose pdf is $f_{X_t}(x)=\frac{f(x+t)}{\overline F(t)}$ and survival function $\overline F_{X_t}(x)=\frac{\overline F(x+t)}{\overline F(t)}$, $x>0$, called the residual extropy (REx) at time t and defined as
\begin{eqnarray}
\nonumber
J(X_t)=-\frac{1}{2(\overline{F}(t))^2} \int_{t}^{+\infty}{f^2(x)}\,dx,
\end{eqnarray}
where $ \overline{F}(t)=\mathbb P(X>t)=1-F(t) $ is the survival (reliability) function of $X$.

Krishnan et al. \cite{Krishnan} and Kamari and Buono \cite{Kamari} studied the dual measure of residual extropy for past lifetime $_tX=(X \mid X\leq t)$, whose pdf is $f_{_tX}(x)=\frac{f(x)}{F(t)}$ and cumulative distribution function $F_{_t{}X}(x)=\frac{F(x)}{F(t)}$, $0<x<t$, called past extropy (PEx) and defined as follows:
\begin{eqnarray}
\nonumber
J(_tX)=-\frac{1}{2({F}(t))^2} \int_{0}^{t}{f^2(x)}\,dx.
\end{eqnarray}
Recently, there has been growing attention to study uncertainty measures for doubly truncated random variable which is widely applied in many fields such as survival analysis and reliability engineering. In survival analysis, if the lifetime of the item falls in an interval $(t_1,t_2)$, information about lifetime between these two points (also named doubly truncated failure time) is studied, see, for instance, Betensky and Martin \cite{Betensky}, Khorashadizadeh et al. \cite{Khorashadizadeh} and Poursaeed and Nematollahi \cite{Poursaeed}. Then, the random variable $(X\mid t_1<X<t_2)$ is introduced with pdf $f_{X_{t_1,t_2}}(x)=\frac{f(x)}{F(t_2)-F(t_1)}$ and cdf $F_{X_{t_1,t_2}}(x)=\frac{F(x)-F(t_1)}{F(t_2)-F(t_1)}$, $t_1<x<t_2$. With this motivation, Sunoj et.al. \cite{Sunoj} introduced Interval Entropy to measure uncertainty in truncated random variable $(X\mid {t_1<X<t_2})$ that is defined as follows:
\begin{eqnarray}
H(t_1,t_2)=-\int_{t_1}^{t_2} \frac{f(x)}{{F}(t_2)-{F}(t_1)} \log \frac{f(x)}{{F}(t_2)-{F}(t_1)}    \,dx, \label{Interval entropy}
\end{eqnarray}
which is an extension of Shannon Entropy. If $t_2\rightarrow+\infty$, then $H(t_1,t_2)$ tends to residual entropy which was introduced by Ebrahimi \cite{Ebrahimi}. Moreover, if $t_1\rightarrow0$, then $H(t_1,t_2)$ tends to the past entropy defined by Di Crescenzo and Longobardi \cite{Di Crescenzo}. Several other properties of the interval entropy were studied by Misagh and Yari \cite{Misagh 2012}.

Di Crescenzo and Longobardi \cite{Di Crescenzo2006} defined weighted entropy, weighted residual entropy and weighted past entropy, which are respectively given by
\begin{eqnarray}
\nonumber
H^w(X)&=&-\int_{0}^{+\infty}{xf(x)}\log f(x)\,dx, \\
\nonumber
H^w(X_t)&=&-\int_{0}^{+\infty}{x\frac{f(x)}{\overline{F}(t)}}\log \frac{f(x)}{\overline{F}(t)}\,dx, \\
\nonumber
H^w(_tX)&=&-\int_{0}^{+\infty}{x\frac{f(x)}{{F}(t)}}\log \frac{f(x)}{{F}(t)}\,dx.
\end{eqnarray}
Balakrishnan et al. \cite{Balakrishnan} introduced weighted Extropy and its dynamic versions as Weighted Residual Extropy and Weighted Past Extropy for residual and past lifetime as below:
\begin{eqnarray}
\nonumber
J^w(X)&=&-\frac{1}{2} \int_{0}^{+\infty}{x f^2(x)}\,dx, \\
\nonumber
J^w(X_t)&=&-\frac{1}{2(\overline{F}(t))^2} \int_{t}^{+\infty}{x f^2(x)}\,dx, \\
\nonumber
J^w(_{t}X)&=&-\frac{1}{2({F}(t))^2} \int_{0}^{t}{x f^2(x)}\,dx.
\end{eqnarray}
Weighted Interval Entropy was introduced by Misagh and Yari \cite{Misagh 2011} for doubly truncated random variable $(X\mid {t_1<X<t_2})$ as follows:

\begin{eqnarray}
\nonumber
IH^w(t_1,t_2)=-\int_{t_1}^{t_2}{x\frac{f(x)}{{F}(t_2)-{F}(t_1)}} \log \frac{f(x)}{{F}(t_2)-{F}(t_1)}\,dx.
\end{eqnarray}
In analogy with the novel measures of uncertainty (i.e., Interval Entropy and Weighted Interval Entropy), here we introduce the concepts of Interval Extropy and Weighted Interval Extropy for doubly truncated random variables. 

\section{Interval Extropy}
Let us suppose that the random variable $(X\mid {t_1<X<t_2})$ represents the lifetime of a unit which fails between $t_1$ and $t_2$ where  ${(t_1,t_2)}\in D=\{(u,v)\in R_+^2: F(u)<F(v)\}$, the Extropy for the doubly truncated random variable is defined as follows:
\begin{eqnarray}
IJ(t_1,t_2)=IJ(X\mid {t_1<X<t_2})=-\frac{1}{2({F}(t_2)-{F}(t_1))^2} \int_{t_1}^{t_2}{f^2(x)}\,dx, \label{Interval extropy}
\end{eqnarray}
which is an extension of Extropy and is called Interval Extropy (IEx). In (\ref{Interval extropy}) we have omitted the dependence of $X$ in the expression $IJ(t_1,t_2)$, but when it is necessary we denote by $IJ_X(t_1,t_2)$ the interval extropy of $X$ to distinguish it from the interval extropy of another random variable.
\begin{remark}\label{Rmk1} It is clear that $IJ(0,t_2)=J(_{t_2}X)$, $IJ(t_1,+\infty)=J(X_{t_1})$ and $IJ(0,$ $ +\infty)=J(X)$ are Past Extropy, Residual Extropy and Extropy, respectively.
\end{remark}

\begin{example}
\label{ex1}
Let $X\sim Exp(\lambda)$, $\lambda>0$ and support $(0,+\infty)$. Based on (\ref{Interval extropy}), we evaluate the interval extropy of $X$ for $0<t_1<t_2<+\infty$ and we obtain
\begin{eqnarray}
\nonumber
IJ(t_1,t_2)&=&\frac{-1}{2(e^{-\lambda t_1}-e^{-\lambda t_2})^2}\int_{t_1}^{t_2} \lambda^2 e^{-2\lambda x} dx \\
\nonumber
&=& -\frac{\lambda}{4} \cdot \frac{e^{-\lambda t_2}+e^{-\lambda t_1}}{e^{-\lambda t_1}-e^{-\lambda t_2}}.
\end{eqnarray}
In Figure \ref{fig-exp}, we plot the interval extropy as a function of $t_1$ for fixed $t_2$ (Figure \ref{fig-exp}(a)) and vice versa (Figure \ref{fig-exp}(b)) for $\lambda=1$.
\begin{figure}
     \centering
     \begin{subfigure}[b]{0.48\textwidth}
         \centering
         \includegraphics[width=\textwidth]{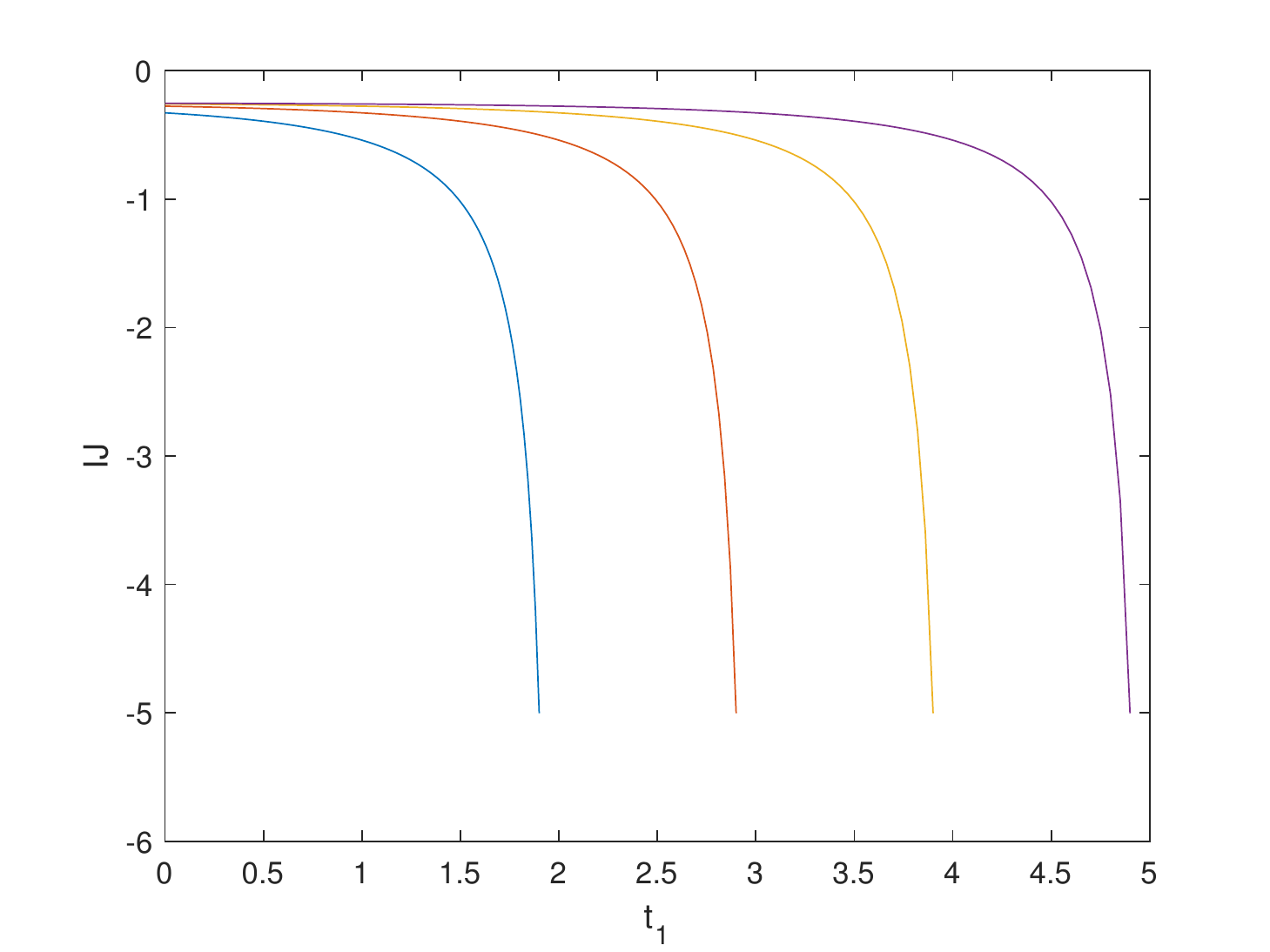}
         \caption{$IJ(t_1,t_2)$ as a function of $t_1$}
     \end{subfigure}
     \hfill
     \begin{subfigure}[b]{0.48\textwidth}
         \centering
         \includegraphics[width=\textwidth]{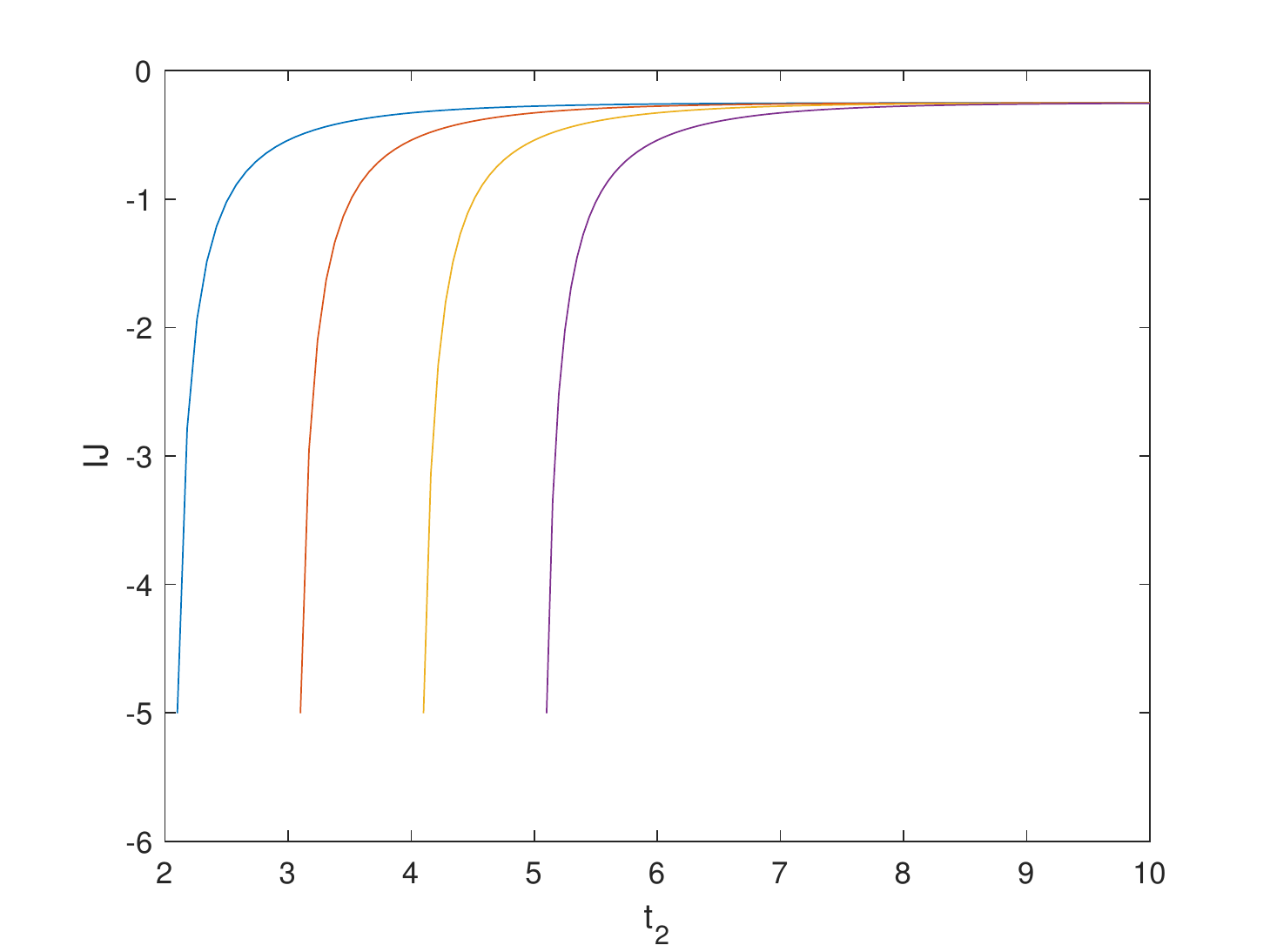}
         \caption{$IJ(t_1,t_2)$ as a function of $t_2$}
     \end{subfigure}
     \hfill
        \caption{Plot of $IJ$ in Example \ref{ex1} as a function of $t_1$ or $t_2$ fixing the other one with  $t_i=2$ (blue), 3 (red), 4 (yellow) and 5 (violet), $i=1,2$.}
        \label{fig-exp}
\end{figure}
\end{example}

\begin{example}
\label{ex2}
Let $X$ follow the Weibull distribution, $W2(\alpha,\lambda)$, with parameters $\alpha=\lambda=2$, $X\sim W2(2,2)$. The cdf and the pdf of $X$ are expressed as
\begin{eqnarray}
\nonumber
F(x)=1-\exp(-2x^2), \ \ \ f(x)=4x\exp(-2x^2), \ \ \ \ x\in(0,+\infty).
\end{eqnarray}
Since the expression of the interval extropy is not given in terms of elementary functions, in Figure \ref{fig-wei2}, we plot the interval extropy as a function of $t_1$ for fixed $t_2$ (Figure \ref{fig-wei2}(a)) and vice versa (Figure \ref{fig-wei2}(b)).
\begin{figure}
     \centering
     \begin{subfigure}[b]{0.48\textwidth}
         \centering
         \includegraphics[width=\textwidth]{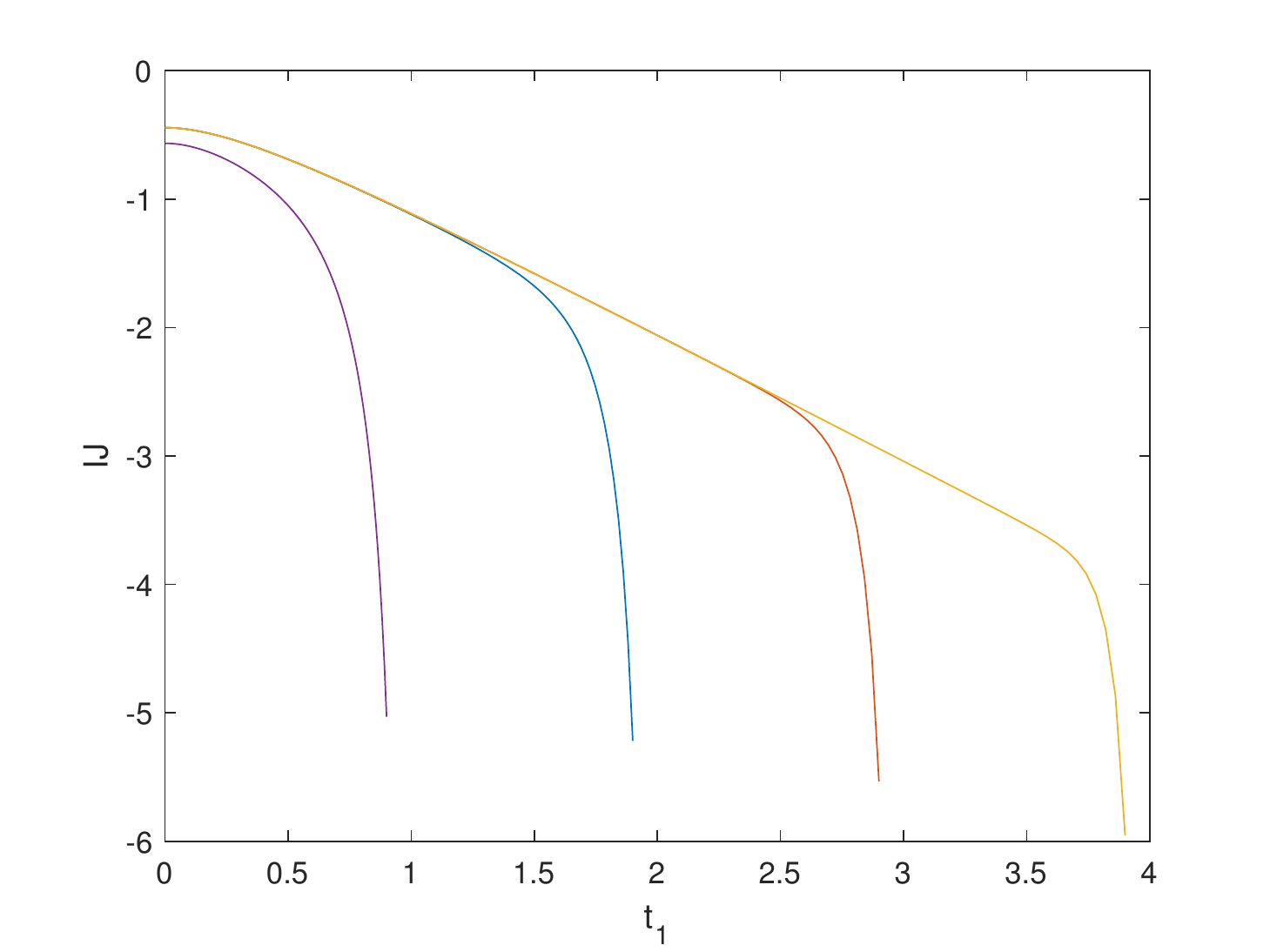}
         \caption{$IJ(t_1,t_2)$ as a function of $t_1$}
     \end{subfigure}
     \hfill
     \begin{subfigure}[b]{0.48\textwidth}
         \centering
         \includegraphics[width=\textwidth]{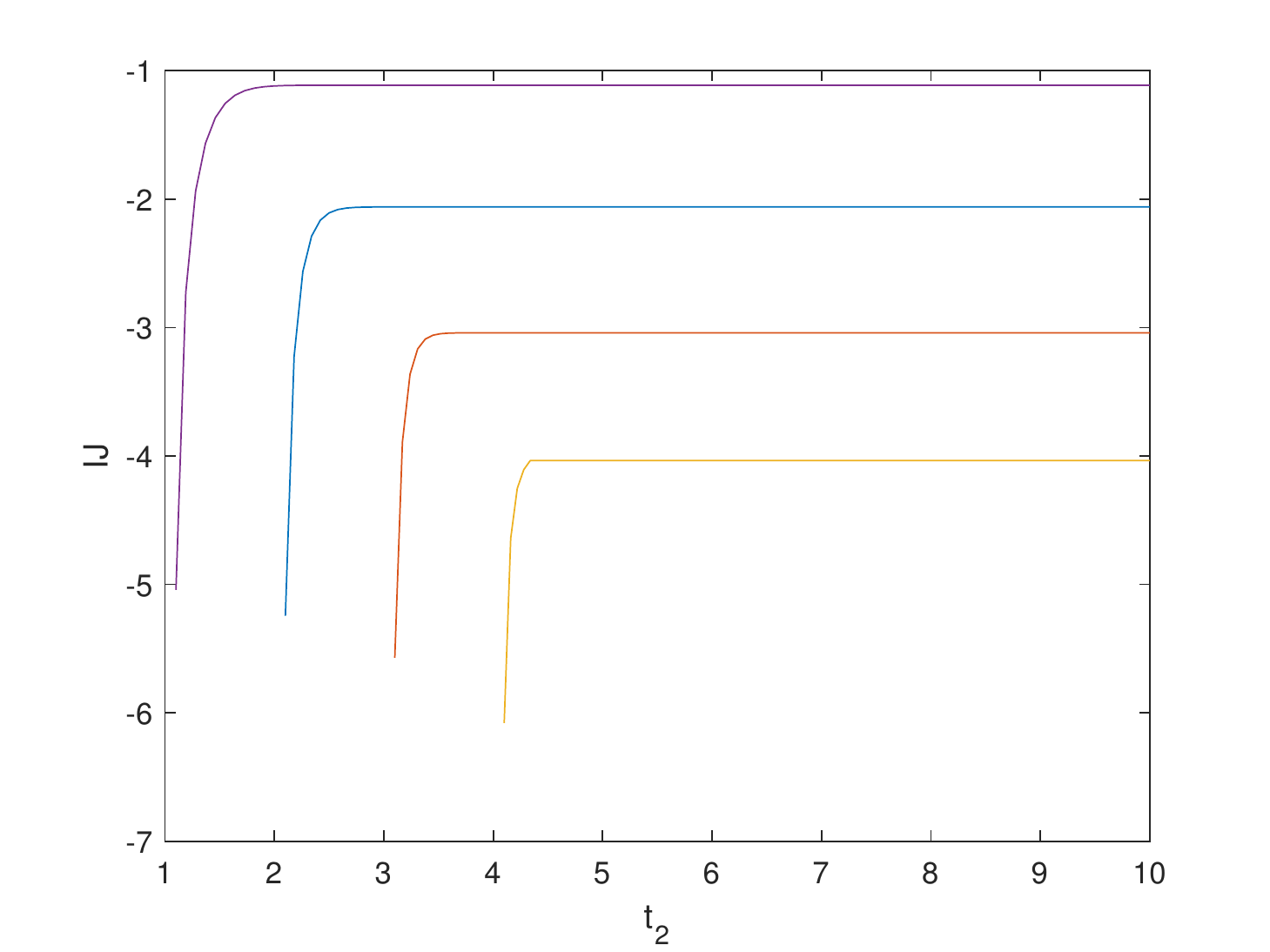}
         \caption{$IJ(t_1,t_2)$ as a function of $t_2$}
     \end{subfigure}
     \hfill
        \caption{Plot of $IJ$ in Example \ref{ex2} as a function of $t_1$ or $t_2$ fixing the other one with  $t_i=1$ (violet), 2 (blue), 3 (red) and 4 (yellow), $i=1,2$.}
        \label{fig-wei2}
\end{figure}
From Figure \ref{fig-wei2}(b) we observe an asymptotic behavior of the interval extropy as $t_2\to+\infty$ towards $-t_1$, i.e., when the interval extropy $IJ(t_1,t_2)$ reduces to the residual extropy $J(X_{t_1})$. In fact, the residual extropy of $X$ in $t$ can be derived as
\begin{eqnarray}
\nonumber
J(X_t)&=& -\frac{1}{2\exp(-4t^2)}\int_t^{+\infty} 16x^2 \exp(-4x^2) dx \\
\nonumber
&=&  -t-\frac{1}{\exp(-4t^2)}\int_t^{+\infty} \exp(-4x^2) dx  \\
\nonumber
&=& -t-\frac{1}{2\sqrt{2}\exp(-4t^2)}\int_{2\sqrt{2}t}^{+\infty} \exp\left(-\frac{y^2}{2}\right)dy=-t-\frac{\sqrt{\pi}}{2}\ \frac{\overline F_Z(2\sqrt{2}t)}{\exp(-4t^2)},
\end{eqnarray}
where $\overline F_Z(\cdot)$ is the survival function of $Z\sim N(0,1)$.
\end{example}

\begin{example}
\label{ex3}
Let $X$ follow the Lognormal distribution, $Lognormal(\mu,\sigma^2)$, with parameters $\mu=0, \ \sigma^2=1$, $X\sim Lognormal(0,1)$. The pdf of $X$ is expressed as
\begin{eqnarray}
\nonumber
f(x)=\frac{1}{x \sqrt{2\pi}}\exp\left(-\frac{\log^2 x}{2}\right), \ \ \ \ x\in(0,+\infty).
\end{eqnarray}
Since the expression of the interval extropy is not given in terms of elementary functions, in Figure \ref{fig-logn}, we plot the interval extropy as a function of $t_1$ for fixed $t_2$ (Figure \ref{fig-logn}(a)) and vice versa (Figure \ref{fig-logn}(b)).
\begin{figure}
     \centering
     \begin{subfigure}[b]{0.48\textwidth}
         \centering
         \includegraphics[width=\textwidth]{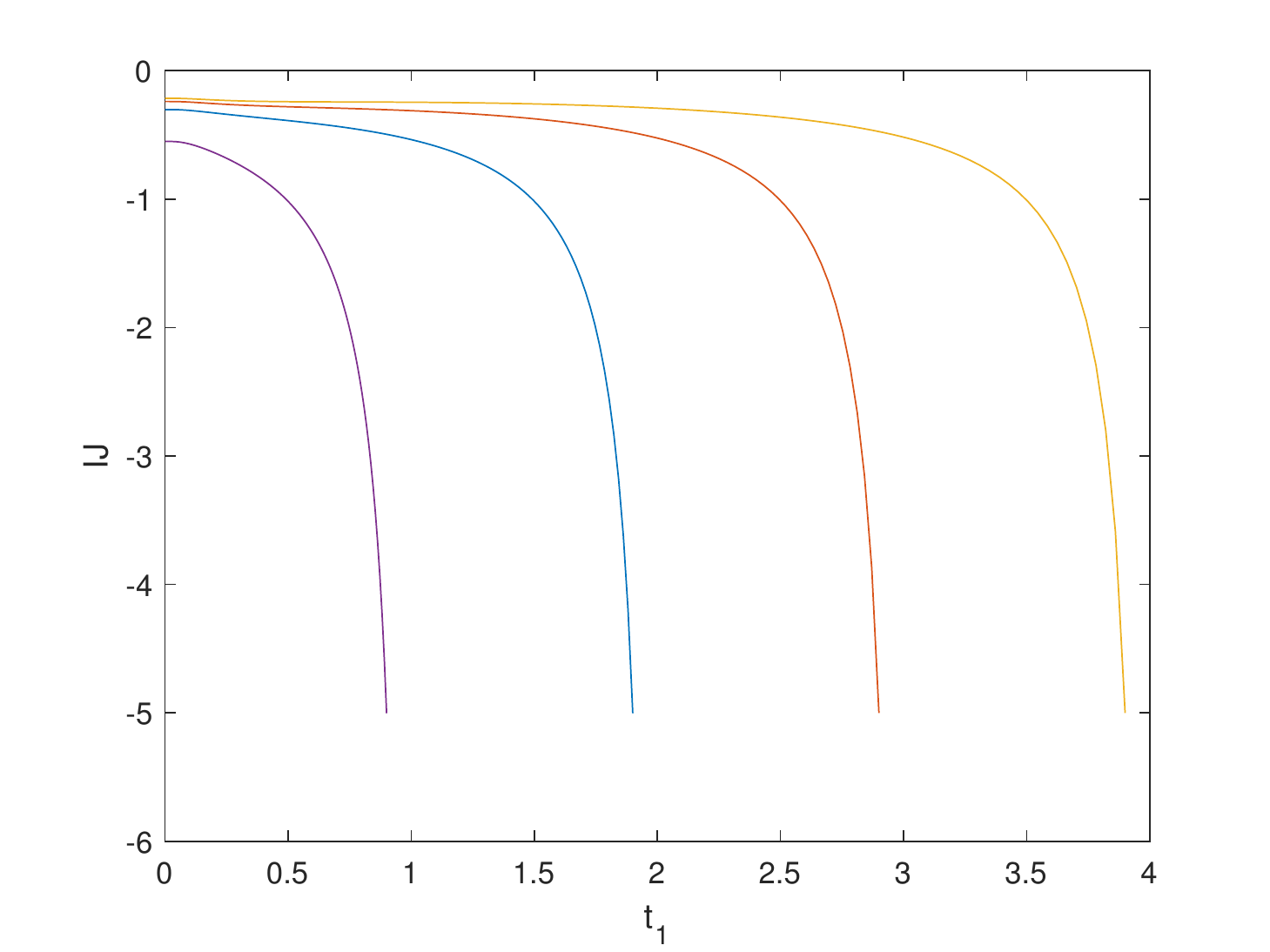}
         \caption{$IJ(t_1,t_2)$ as a function of $t_1$}
     \end{subfigure}
     \hfill
     \begin{subfigure}[b]{0.48\textwidth}
         \centering
         \includegraphics[width=\textwidth]{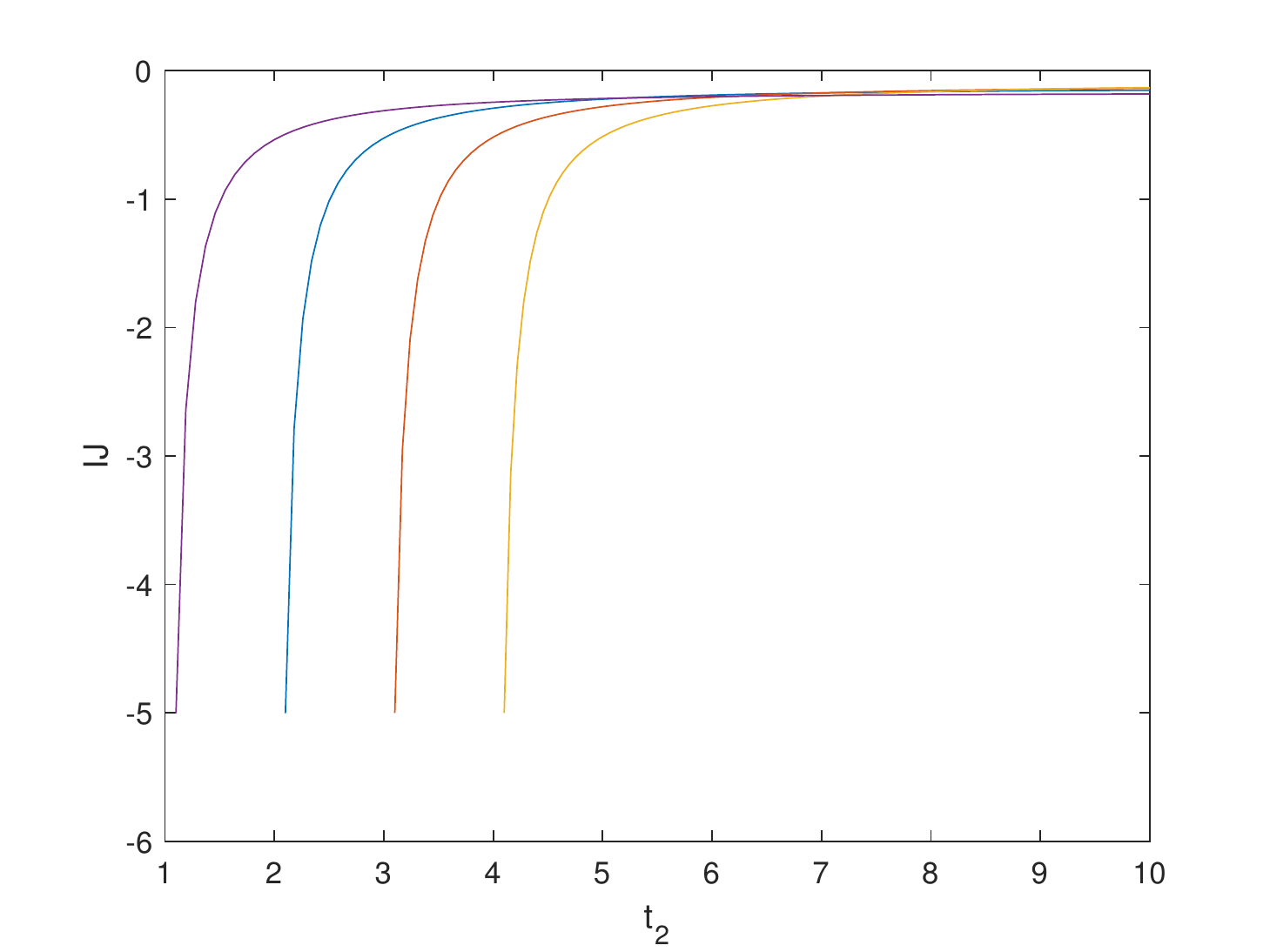}
         \caption{$IJ(t_1,t_2)$ as a function of $t_2$}
     \end{subfigure}
     \hfill
        \caption{Plot of $IJ$ in Example \ref{ex3} as a function of $t_1$ or $t_2$ fixing the other one with  $t_i=1$ (violet), 2 (blue), 3 (red) and 4 (yellow), $i=1,2$.}
        \label{fig-logn}
\end{figure}
\end{example}

Based on the above examples, it could seem that the interval extropy is always decreasing with respect to $t_1$ and always increasing with respect to $t_2$. In the following, we provide two counterexamples to show that the interval extropy can be non monotonous with respect to $t_1$ and $t_2$.

\begin{example}
\label{exnew}
Let $X$ be a random variable with support $S=(a,+\infty)$, $a>0$, whose cdf is defined as $F(x)=1-\left(\frac{a}{x}\right)^b$, $b>0$. The interval extropy of $X$ can be obtained as follows
\begin{eqnarray}
\nonumber
IJ(t_1,t_2)&=&  -\frac{1}{2\left[\left(\frac{a}{t_1}\right)^b-\left(\frac{a}{t_2}\right)^b\right]^2}\int_{t_1}^{t_2} \frac{b^2a^{2b}}{x^{2b+2}}dx \\
\nonumber
&=&  \frac{1}{2\left[\frac{1}{t_1^b}-\frac{1}{t_2^b}\right]^2}\frac{b^2}{2b+1}\left[\frac{1}{t_2^{2b+1}}-\frac{1}{t_1^{2b+1}}\right]    \\
\nonumber
&=&\frac{b^2\left(t_1^{2b+1}-t_2^{2b+1}\right)}{2(2b+1)t_1t_2 \left(t_2^b-t_1^b\right)^2}.
\end{eqnarray}
Let us focus on the case $a=1$ and $b=10$. In Figure \ref{fignew} we have plotted the interval extropy of $X$ as a function of $t_1$ for fixed different values of $t_2$ and we can observe that it is initially increasing and then decreasing with respect to $t_1$.
\begin{figure}
     \centering
         \includegraphics[scale=0.5]{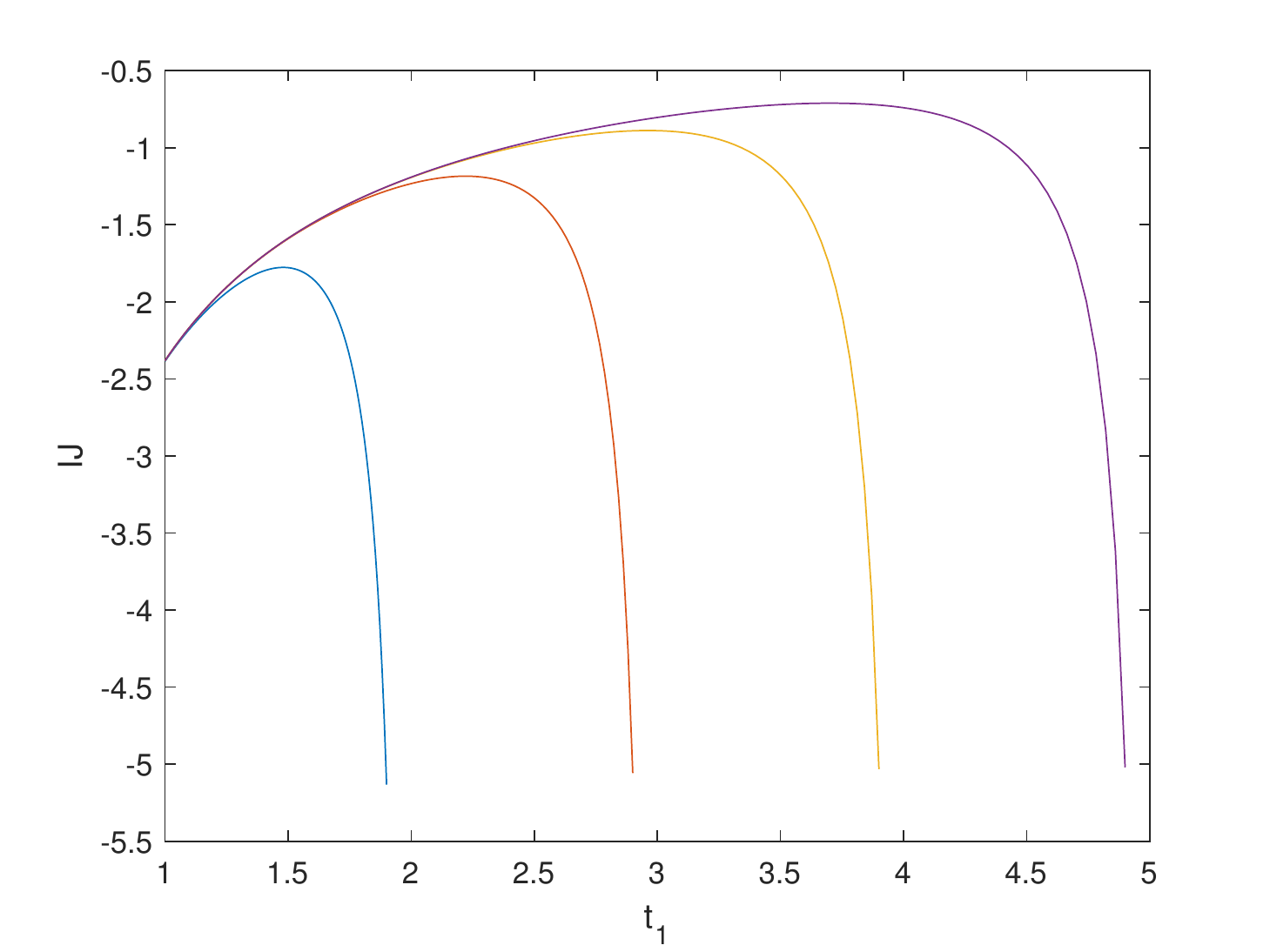}
        \caption{Plot of $IJ$ in Example \ref{exnew} as a function of $t_1$ fixing  $t_2=2$ (blue), 3 (red), 4 (yellow) and 5 (violet).}
        \label{fignew}
\end{figure}
\end{example}

\begin{example}
\label{exnew2}
Let $X$ be a random variable with cdf and pdf respectively defined as
\begin{eqnarray}
\nonumber
F(x)&=&\begin{cases} \exp\left(-\frac{1}{2}-\frac{1}{x}\right), &\mbox{ if } \ \ x\in(0,1] \\
                             \exp\left(-2+\frac{x^2}{2}\right),  &\mbox{ if } \ \ x\in[1,2)
\end{cases}
\\
\nonumber
f(x)&=&\begin{cases} \exp\left(-\frac{1}{2}-\frac{1}{x}\right)\frac{1}{x^2}, &\mbox{ if } \ \ x\in(0,1] \\
                             \exp\left(-2+\frac{x^2}{2}\right)x,  &\mbox{ if } \ \ x\in[1,2).
\end{cases}
\end{eqnarray}
In Figure \ref{fignew2} we have plotted the interval extropy as a function of $t_2\in(1,2)$ with fixed $t_1=0.1$ and we can observe a non monotonic behavior.
\begin{figure}
     \centering
         \includegraphics[scale=0.5]{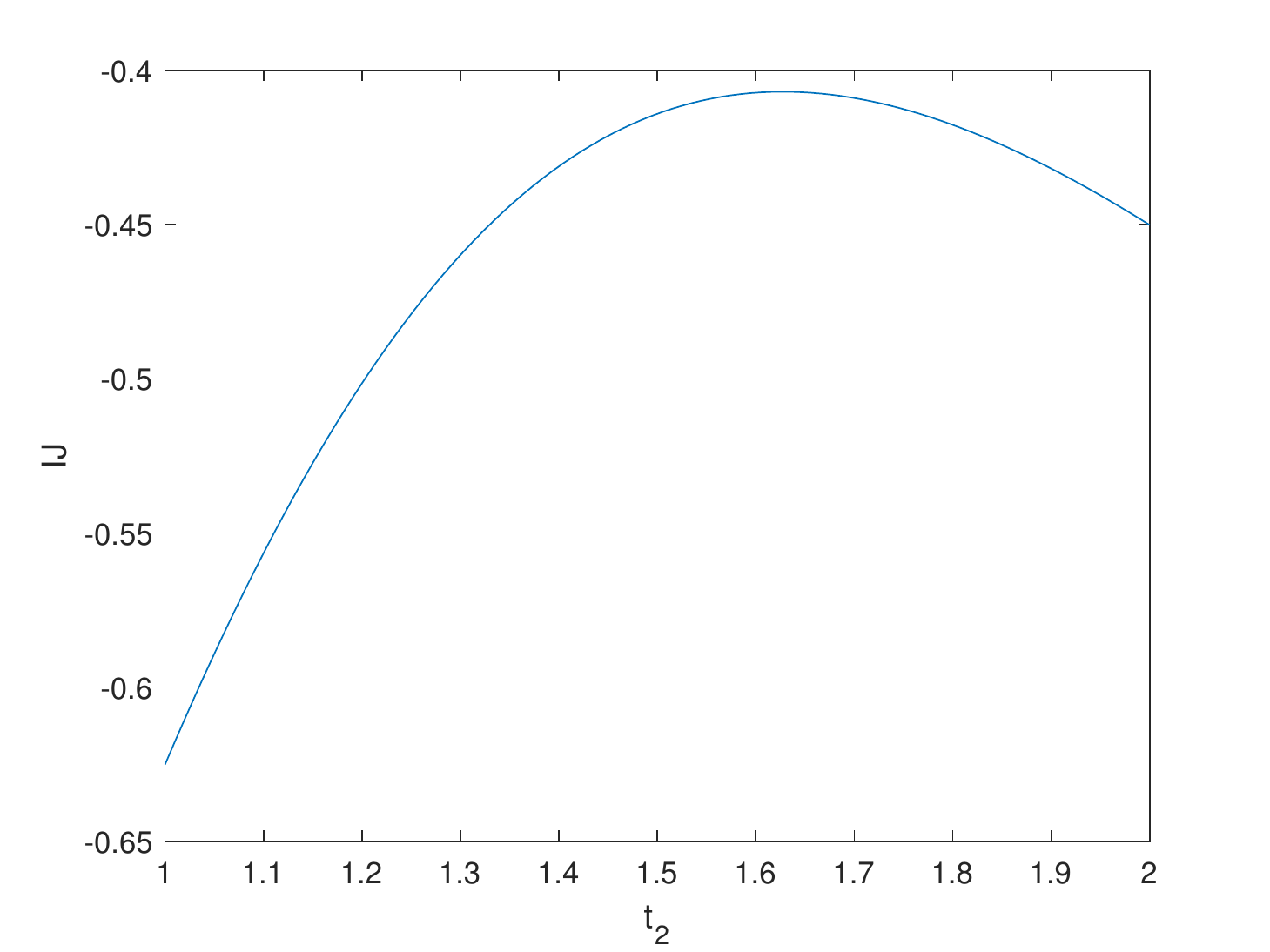}
        \caption{Plot of $IJ$ in Example \ref{exnew2} as a function of $t_2\in(1,2)$ fixing  $t_1=0.1$.}
        \label{fignew2}
\end{figure}
\end{example}

In the following theorem, we show the connection among Extropy and its dynamic versions.
\begin{theorem}
\label{thm1}
Let $X$ be a random variable denoting the lifetime of a component. For all $0<t_1<t_2<+\infty$ the extropy can be decomposed as follows:
\begin{eqnarray}
J(X)=F^2(t_1) J(_{t_1}X)+({F}(t_2)-{F}(t_1))^2 IJ(t_1,t_2)+ \overline{F}^2(t_2) J(X_{t_2}), \label{decomposition}
\end{eqnarray}
i.e., the Extropy is a function of PEx, REx and IEx. 
\end{theorem}

\begin{proof}
From the definition of extropy (\ref{extropy}), we can write
\begin{eqnarray}
\nonumber
J(X)&=&-\frac{1}{2} \int_{0}^{+\infty}{f^2(x)}\,dx \\
\label{eq14}
&=&-\frac{1}{2} \int_{0}^{t_1}{f^2(x)}\,dx -\frac{1}{2} \int_{t_1}^{t_2}{f^2(x)}\,dx -\frac{1}{2} \int_{t_2}^{+\infty}{f^2(x)}\,dx.
\end{eqnarray}
Now, we observe that the terms in the RHS of (\ref{eq14}) are related to past extropy, interval extropy and residual extropy as
\begin{eqnarray}
\nonumber
-\frac{1}{2} \int_{0}^{t_1}{f^2(x)}\,dx &=& F^2(t_1) J(_{t_1}X),\\
\nonumber
-\frac{1}{2} \int_{t_1}^{t_2}{f^2(x)}\,dx &=&({F}(t_2)-{F}(t_1))^2 IJ(t_1,t_2), \\
\nonumber
-\frac{1}{2} \int_{t_2}^{+\infty}{f^2(x)}\,dx &=& \overline{F}^2(t_2) J(X_{t_2}),
\end{eqnarray}
and then we obtain the stated result.
\end{proof}

Relation (\ref{decomposition}) shows that the uncertainty about the failure time of a component consists of 3 parts:
\begin{itemize}
\item[$1^{st}$.] The uncertainty of the failure time in $(0,t_1)$; 
\item[$2^{nd}$.] The uncertainty of the failure time in $(t_1,t_2)$; 
\item[$3^{rd}$.] The uncertainty about the failure time in $(t_2,+\infty)$.
\end{itemize}

The corresponding aging classes are defined as follows.

\begin{definition}
\label{defdec}
 The random variable $X$ is said to have decreasing $IJ$ property if and only if for any fixed $t_2$, $IJ(t_1,t_2)$ is decreasing respect to $t_1$.
\end{definition}
\begin{definition}
\label{definc}
 The random variable $X$ is said to have increasing  $IJ$ property if and only if for any fixed $t_1$, $IJ(t_1,t_2)$ is increasing respect to $t_2$.
\end{definition}

\begin{remark}
As seen in Example \ref{ex1}, the exponential distribution satisfies both the conditions in Definitions \ref{defdec} and \ref{definc}.
\end{remark}

\begin{definition}
Let $X$ be a non-negative and absolutely continuous random variable with cdf $F$ and pdf $f$. The Generalized Failure Rate (GFR) functions of $X$ in $t_1$ and $t_2$ (with $F(t_2)-F(t_1)>0$) are defined in \cite{Navarro} as
\begin{equation}
h_i(t_1,t_2)=\frac{f(t_i)}{F(t_2)-F(t_1)}, \ \  \ \ i=1,2.
\end{equation}
\end{definition}

An upper bound in terms of GFR is obtained for Interval Extropy in the following theorem.
\begin{theorem}
\label{thmub}
Let $X$ be an absolutely continuous non-negative random variable. If IJ is increasing in $t_2$, then
\begin{eqnarray}
IJ(t_1,t_2)\leq -\frac{h_2(t_1,t_2)}{4}. \label{bound of IE2}
\end{eqnarray}
\end{theorem}

\begin{proof}
By differentiating IEx with respect to $t_2$, we have
\begin{eqnarray} 
\frac{\partial IJ(t_1,t_2)}{\partial t_2}&=& -\frac{h^2_2(t_1,t_2)}{2}-2h_2(t_1,t_2)IJ(t_1,t_2). \label{derivative to t2}
\end{eqnarray}
If $IJ(t_1,t_2)$ is increasing in $t_2$ then (\ref{derivative to t2}) implies (\ref{bound of IE2}).
\end{proof}

In the following proposition, we analyze the effect of a linear transformation on the interval extropy.

\begin{proposition}
\label{pr1}
Let $X$ be a non negative and absolutely continuous random variable and let $Y=aX+b$ where $a>0$ and $b\ge0$. The interval extropy of $Y$ is given in terms of the interval extropy of $X$ as
\begin{equation}
\label{linear}
IJ_Y(t_1,t_2)=\frac{1}{a}IJ_X\left(\frac{t_1-b}{a},\frac{t_2-b}{a}\right),
\end{equation}
where $t_1,t_2\in S_Y$.
\end{proposition}

\begin{proof}
The cdf and the pdf of $Y$ can be expressed in terms of $F_X$ and $f_X$ as
\begin{equation}
\label{utu}
F_Y(x)=F_X\left(\frac{x-b}{a}\right), \ \ \ f_Y(x)=\frac{1}{a}f_X\left(\frac{x-b}{a}\right).
\end{equation}
Hence, the interval extropy of $Y$ can be expressed as
\begin{eqnarray}
\nonumber
IJ_Y(t_1,t_2)&=&-\frac{1}{2\left(F_X\left(\frac{t_2-b}{a}\right)-F_X\left(\frac{t_1-b}{a}\right)\right)^2}\int_{t_1}^{t_2} \frac{1}{a^2}f_X^2\left(\frac{x-b}{a}\right)\,dx \\
\nonumber
&=& -\frac{1}{2\left(F_X\left(\frac{t_2-b}{a}\right)-F_X\left(\frac{t_1-b}{a}\right)\right)^2}\int_{\frac{t_1-b}{a}}^{\frac{t_2-b}{a}} \frac{1}{a}f_X^2(x)\,dx \\
\nonumber
&=&\frac{1}{a}IJ_X\left(\frac{t_1-b}{a},\frac{t_2-b}{a}\right),
\end{eqnarray}
which completes the proof.
\end{proof}

In the following theorem, we give a characterization of the exponential distribution based on the interval extropy.

\begin{theorem}
Let $X$ be a random variable with support $(0,+\infty)$, differentiable and strictly positive pdf $f$ and cdf $F$. Then, $X$ is exponentially distributed if, and only if, for all $(t_1,t_2)$ such that $ 0<t_1<t_2<+\infty$, the following relation holds
\begin{equation}
\label{eqcaex}
IJ(t_1,t_2)=-\frac{1}{4} \left[h_1(t_1,t_2)+h_2(t_1,t_2)\right].
\end{equation}
\end{theorem}

\begin{proof}
Let us suppose $X\sim Exp(\lambda)$. In Example \ref{ex1}, we have evaluated the interval extropy that is given by
\begin{eqnarray}
\nonumber
IJ(t_1,t_2)= -\frac{\lambda}{4} \cdot \frac{e^{-\lambda t_2}+e^{-\lambda t_1}}{e^{-\lambda t_1}-e^{-\lambda t_2}}.
\end{eqnarray}
Moreover, about GFR functions of $X$, we have
\begin{eqnarray}
\nonumber
h_1(t_1,t_2)&=&\frac{\lambda e^{-\lambda t_1}}{e^{-\lambda t_1}-e^{-\lambda t_2}}, \\
\nonumber
h_2(t_1,t_2)&=&\frac{\lambda e^{-\lambda t_2}}{e^{-\lambda t_1}-e^{-\lambda t_2}},
\end{eqnarray}
and then the first part of the proof is completed.

Conversely, let us suppose (\ref{eqcaex}) holds. Then, by making explicit the interval extropy and GFR functions, we obtain
\begin{eqnarray}
\nonumber
-\frac{1}{2({F}(t_2)-{F}(t_1))^2} \int_{t_1}^{t_2}{f^2(x)}\,dx= -\frac{f(t_1)+f(t_2)}{4(F(t_2)-F(t_1))}.
\end{eqnarray}
From the above equation, we have
\begin{equation}
\label{eqpr1}
\int_{t_1}^{t_2}{f^2(x)}\,dx=\frac{1}{2}(F(t_2)-F(t_1))(f(t_1)+f(t_2)).
\end{equation}
By differentiating both sides of (\ref{eqpr1}) with respect to $t_1$, we get
\begin{eqnarray}
\nonumber
-f^2(t_1)=-\frac{1}{2}f(t_1)(f(t_1)+f(t_2))+\frac{1}{2}f'(t_1)(F(t_2)-F(t_1)),
\end{eqnarray}
which reduces to
\begin{equation}
\label{eqpr2}
-f^2(t_1)+f(t_1)f(t_2)=f'(t_1)(F(t_2)-F(t_1)).
\end{equation}
By differentiating both sides of (\ref{eqpr2}) with respect to $t_2$, we get
\begin{eqnarray}
\nonumber
f(t_1)f'(t_2)=f'(t_1)f(t_2),
\end{eqnarray}
which is equivalent to
\begin{eqnarray}
\nonumber
\frac{f'(t_1)}{f(t_1)}=\frac{f'(t_2)}{f(t_2)},
\end{eqnarray}
i.e., the ratio is constant for $x>0$,
\begin{equation}
\label{eqpr3}
\frac{f'(x)}{f(x)}=A.
\end{equation}
Hence, by integrating both sides of (\ref{eqpr3}) from $0$ to $t$, we get
\begin{eqnarray}
\nonumber
f(t)=f(0)\ e^{At},
\end{eqnarray}
and in order to satisfy the condition of normalization for the pdf $f$, we need $A=-f(0)$, i.e., $f$ is the pdf of an exponential distribution.
\end{proof}

\begin{remark}
The equilibrium random variable $Y$ associated to a renewal process with lifetime $X$ is a random variable of primary interest in the context of reliability theory, as pointed out in Barlow and Proschan \cite{barlow}. The survival function and the probability density function of $Y$ are expressed as
\begin{eqnarray}
\nonumber
\overline F_Y(t)= \frac{1}{\mathbb E(X)} \int_t^{+\infty} \overline F_X(x)\,dx, \ \ \ 
f_Y(t)=\frac{\overline{F}_X(t)}{\mathbb E(X)}
\end{eqnarray}
where $\mathbb E(X)$ is the expectation of $X$. We can define Extropy and its interval version of $Y$ as follows:
\begin{eqnarray}
\nonumber
J(Y)=-\frac{1}{2 \mathbb E^{2}(X)} \int_{0}^{\infty}{\overline{F}_X^2(x)}\,dx,
\end{eqnarray}
\begin{eqnarray}
\nonumber
IJ_Y(t_1,t_2)=-\frac{1}{2} \frac{\int_{t_1}^{t_2}{\overline F_X^2(x)}\,dx}{\left(\int_{t_1}^{t_2}{\overline F_X(x)}\,dx\right)^2}.
\end{eqnarray}
\end{remark}

\section{Weighted Interval Extropy}
In order to give importance to the value assumed by the random variable, it is significant to introduce weighted versions of the measures of uncertainty. In fact, most of the well-known measure of discrimination are position-free, in the sense that they assume the same value for $X$ and $X+b$ for any $b\in\mathbb R$. In Proposition \ref{pr1}, we have showed that the interval extropy does not change under translations since, for $Y=X+b$, we have $IJ_Y(t_1+b,t_2+b)=IJ_X(t_1,t_2)$. In this section, we will introduce and study the weighted version of the interval extropy and we will show that it is not invariant under translations. 

Suppose $X$ be a non-negative absolutely continuous random variable. For all $t_1$ and $t_2$ such that ${(t_1,t_2)}\in D={\lbrace(u,v)\in R_+^2: F(u)<F(v)\rbrace}$ we define the Weighted Interval Extropy (WIEx) of $X$ as
\begin{eqnarray}
IJ^w(t_1,t_2)=-\frac{1}{2({F}(t_2)-{F}(t_1))^2} \int_{t_1}^{t_2}{x f^2(x)}\,dx, \label{Weighted Interval extropy}
\end{eqnarray}
in the same way in which in \cite{Misagh 2011} the weighted interval entropy has been defined.

\begin{remark} We notice that 
$$\lim _{t_1\to 0}IJ^w(t_1,t_2)= IJ^w(_{t_2}X) \mbox{  and  } \lim _{t_2\to +\infty}IJ^w(t_1,t_2)= IJ^w(X_{t_1}),$$
where $IJ^w(_{t_2}X)$ and $IJ^w(X_{t_1})$ are Weighted Past Extropy at time $t_2$ and Weighted Residual Extropy at time $t_1$, respectively.
\end{remark}

\begin{example}
\label{ex4}
Let $X\sim Exp(\lambda)$, $\lambda>0$. Based on (\ref{Weighted Interval extropy}), we evaluate the weighted interval extropy of $X$ for $0<t_1<t_2<+\infty$ and we obtain
\begin{eqnarray}
\nonumber
IJ^w(t_1,t_2)&=&\frac{-1}{2(e^{-\lambda t_1}-e^{-\lambda t_2})^2}\int_{t_1}^{t_2} x \lambda^2 e^{-2\lambda x} dx \\
\nonumber
&=& -\frac{\lambda}{4} \cdot \frac{t_1 e^{-2\lambda t_1}-t_2e^{-2\lambda t_2}}{(e^{-\lambda t_1}-e^{-\lambda t_2})^2}-\frac{1}{8} \cdot \frac{e^{-\lambda t_2}+e^{-\lambda t_1}}{e^{-\lambda t_1}-e^{-\lambda t_2}}.
\end{eqnarray}
In Figure \ref{fig-expw}, we plot the interval extropy as a function of $t_1$ for fixed $t_2$ (Figure \ref{fig-expw}(a)) and vice versa (Figure \ref{fig-expw}(b)) for $\lambda=1$.
\begin{figure}
     \centering
     \begin{subfigure}[b]{0.48\textwidth}
         \centering
         \includegraphics[width=\textwidth]{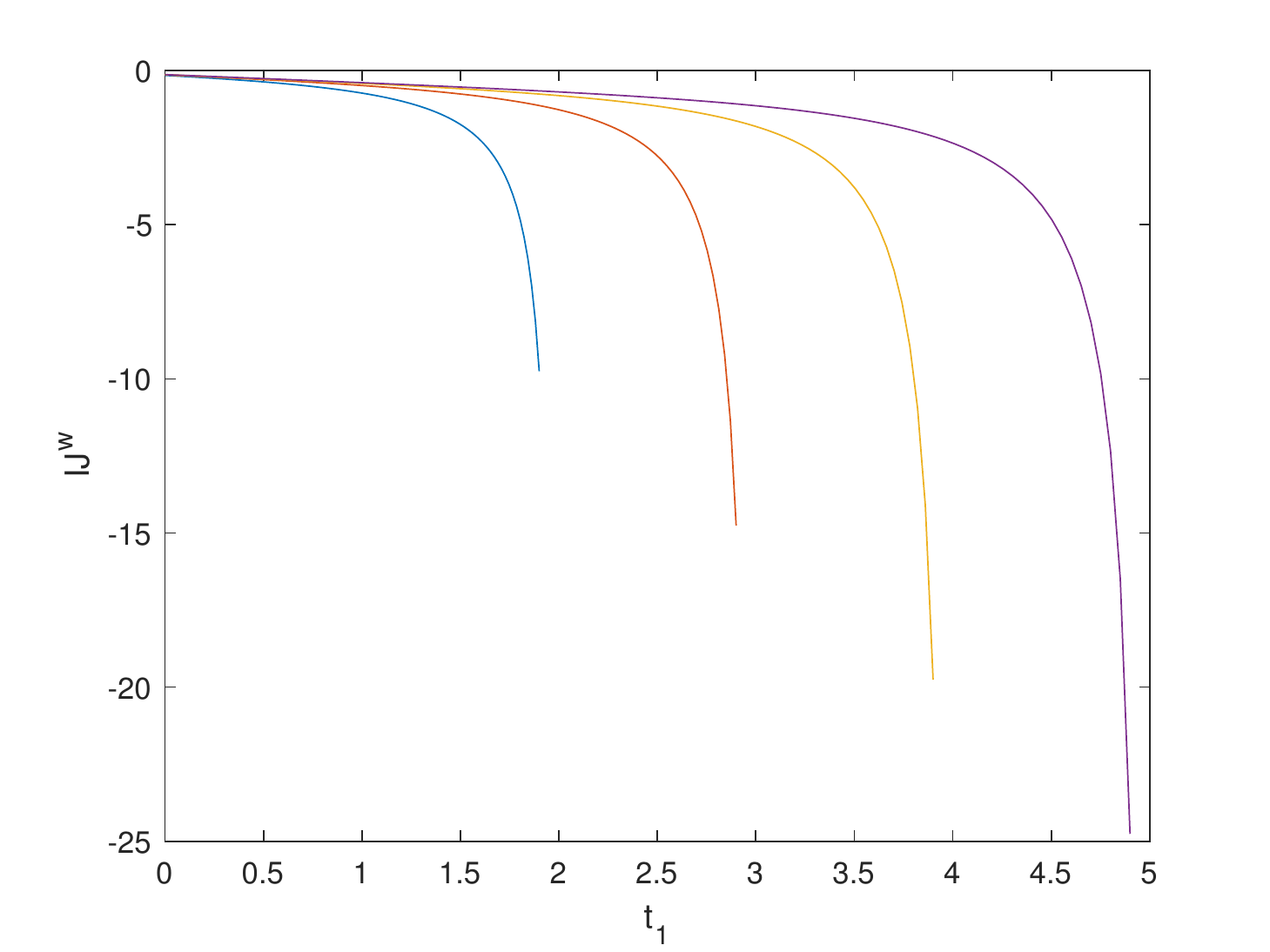}
         \caption{$IJ^w(t_1,t_2)$ as a function of $t_1$}
     \end{subfigure}
     \hfill
     \begin{subfigure}[b]{0.48\textwidth}
         \centering
         \includegraphics[width=\textwidth]{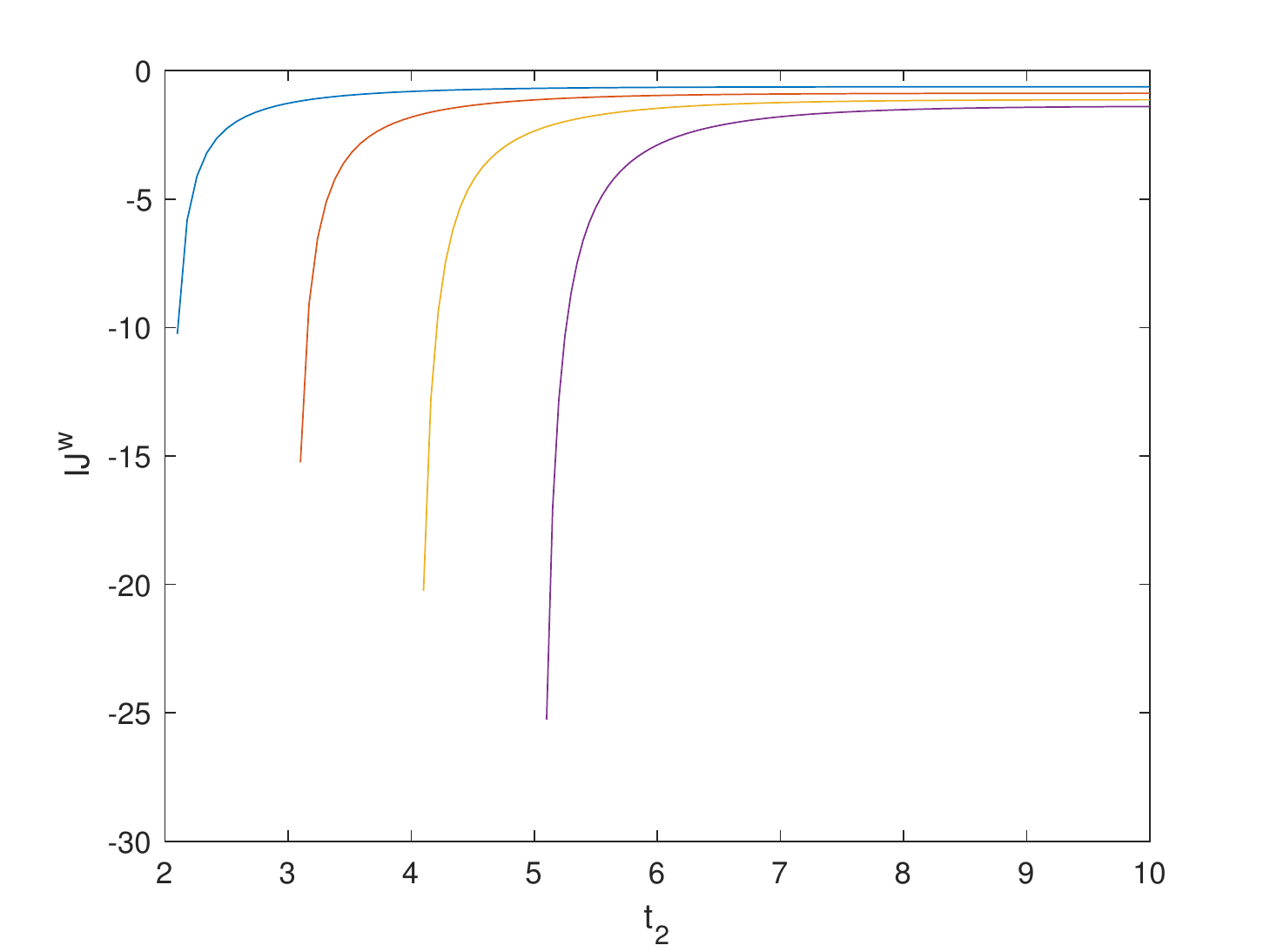}
         \caption{$IJ^w(t_1,t_2)$ as a function of $t_2$}
     \end{subfigure}
     \hfill
        \caption{Plot of $IJ^w$ in Example \ref{ex4} as a function of $t_1$ or $t_2$ fixing the other one with  $t_i=2$ (blue), 3 (red), 4 (yellow) and 5 (violet), $i=1,2$.}
        \label{fig-expw}
\end{figure}
\end{example}

\begin{example}
\label{ex5}
Let $X$ follow the Weibull distribution with parameters $\alpha=\lambda=2$, $X\sim W2(2,2)$. Based on (\ref{Weighted Interval extropy}), we evaluate the weighted interval extropy of $X$ for $0<t_1<t_2<+\infty$ and we obtain
\begin{eqnarray}
\nonumber
IJ^w(t_1,t_2)&=&\frac{-1}{2(\exp(-2 t_1^2)-\exp(-2 t_2^2))^2}\int_{t_1}^{t_2} 16 x^3 \exp(-4x^2) dx \\
\nonumber
&=&  \frac{t_2^2 \exp(-4 t_2^2)-t_1^2 \exp(-4 t_1^2)}{(\exp(-2 t_1^2)-\exp(-2 t_2^2))^2}-\frac{1}{4} \cdot \frac{\exp(-2 t_1^2)+\exp(-2 t_2^2)}{\exp(-2 t_1^2)-\exp(-2 t_2^2)}.
\end{eqnarray}
In Figure \ref{fig-wei2w}, we plot the interval extropy as a function of $t_1$ for fixed $t_2$ (Figure \ref{fig-wei2w}(a)) and vice versa (Figure \ref{fig-wei2w}(b)).
\begin{figure}
     \centering
     \begin{subfigure}[b]{0.48\textwidth}
         \centering
         \includegraphics[width=\textwidth]{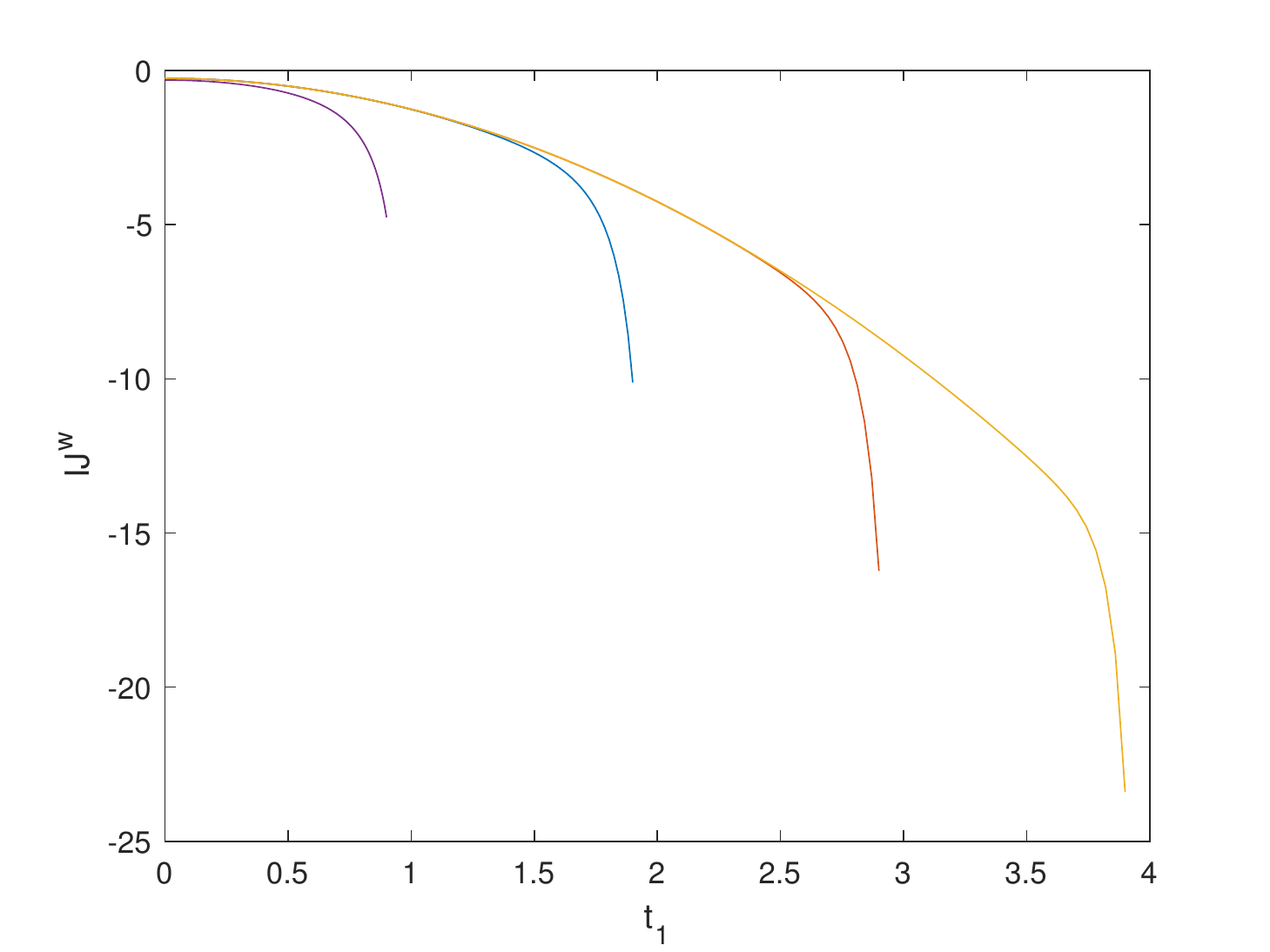}
         \caption{$IJ^w(t_1,t_2)$ as a function of $t_1$}
     \end{subfigure}
     \hfill
     \begin{subfigure}[b]{0.48\textwidth}
         \centering
         \includegraphics[width=\textwidth]{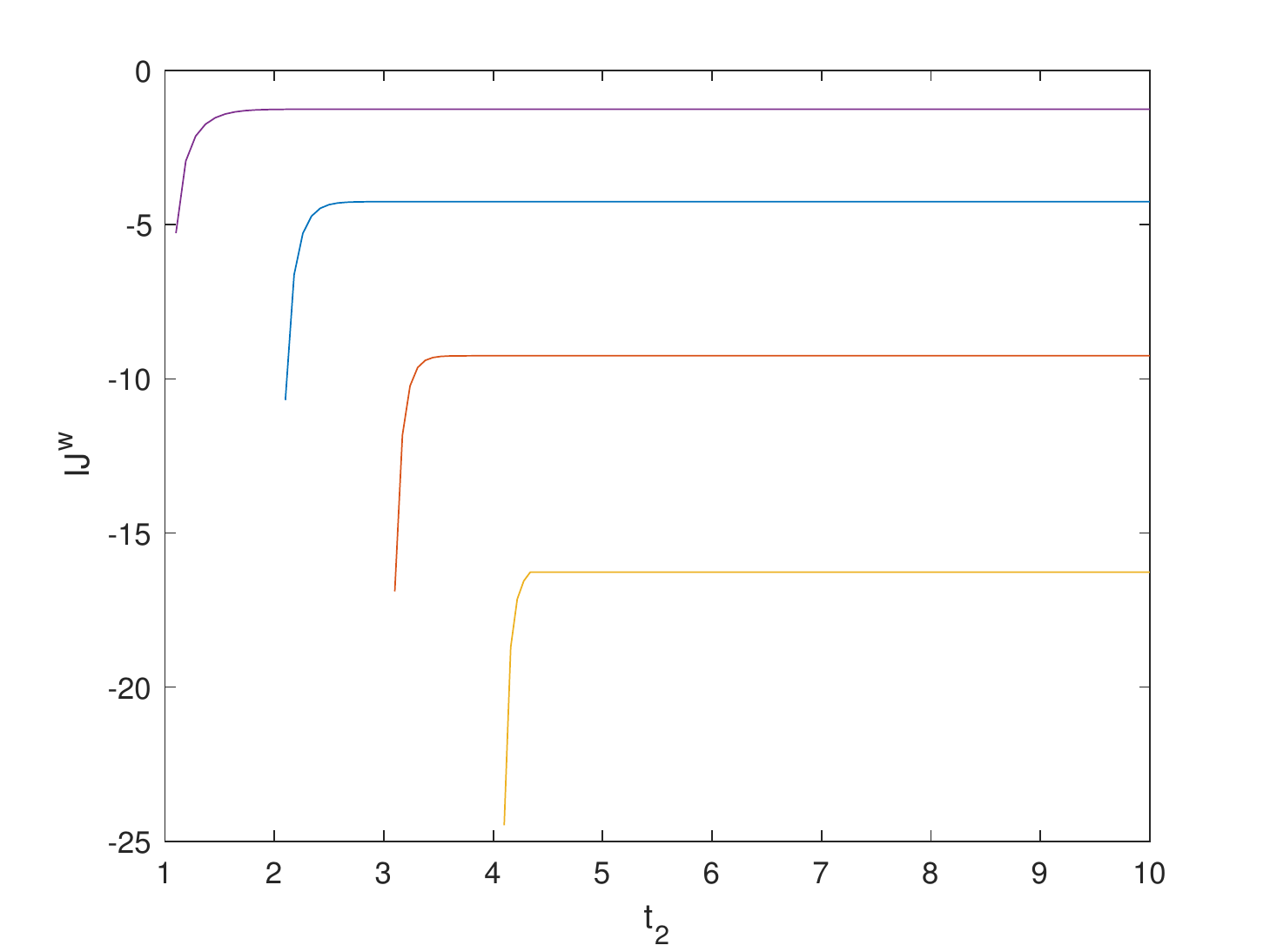}
         \caption{$IJ^w(t_1,t_2)$ as a function of $t_2$}
     \end{subfigure}
     \hfill
        \caption{Plot of $IJ^w$ in Example \ref{ex5} as a function of $t_1$ or $t_2$ fixing the other one with  $t_i=1$ (violet), 2 (blue), 3 (red) and 4 (yellow), $i=1,2$.}
        \label{fig-wei2w}
\end{figure}
From Figure \ref{fig-wei2w}(b) we observe an asymptotic behavior of the weighted interval extropy as $t_2\to+\infty$, i.e., when the weighted interval extropy $IJ^w(t_1,t_2)$ reduces to the weighted residual extropy $J^w(X_{t_1})$. In fact, the weighted residual extropy of $X$ in $t$ can be expressed as
\begin{eqnarray}
\nonumber
J(X_t)=-t^2-\frac{1}{4}.
\end{eqnarray}
\end{example}

\begin{example}
\label{ex6}
Let $X$ follow the Lognormal distribution with parameters $\mu=0, \ \sigma^2=1$, $X\sim Lognormal(0,1)$. Since the expression of the weighted interval extropy is not given in terms of elementary functions, in Figure \ref{fig-lognw}, we plot the weighted interval extropy as a function of $t_1$ for fixed $t_2$ (Figure \ref{fig-lognw}(a)) and vice versa (Figure \ref{fig-lognw}(b)).
\begin{figure}
     \centering
     \begin{subfigure}[b]{0.48\textwidth}
         \centering
         \includegraphics[width=\textwidth]{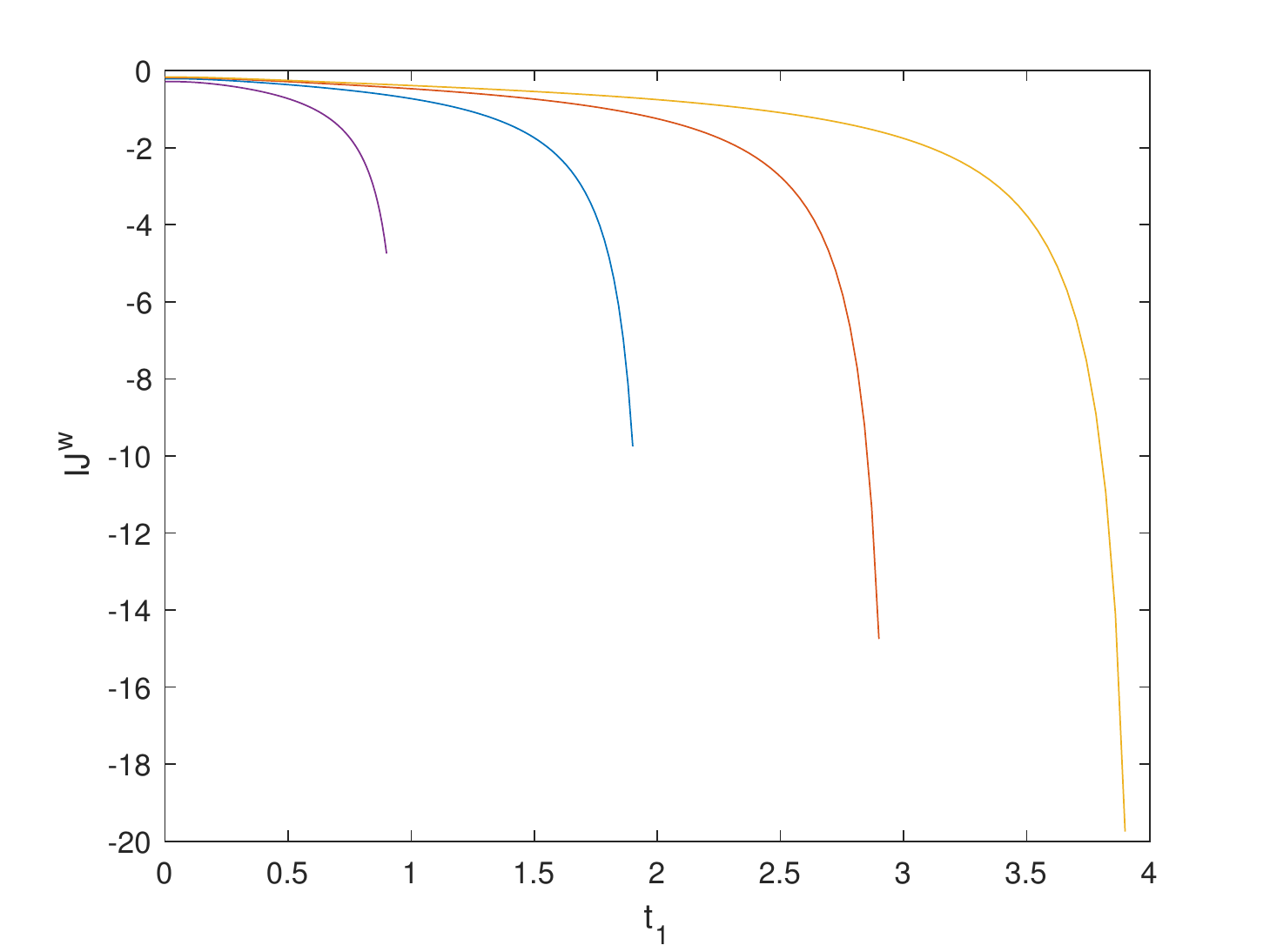}
         \caption{$IJ^w(t_1,t_2)$ as a function of $t_1$}
     \end{subfigure}
     \hfill
     \begin{subfigure}[b]{0.48\textwidth}
         \centering
         \includegraphics[width=\textwidth]{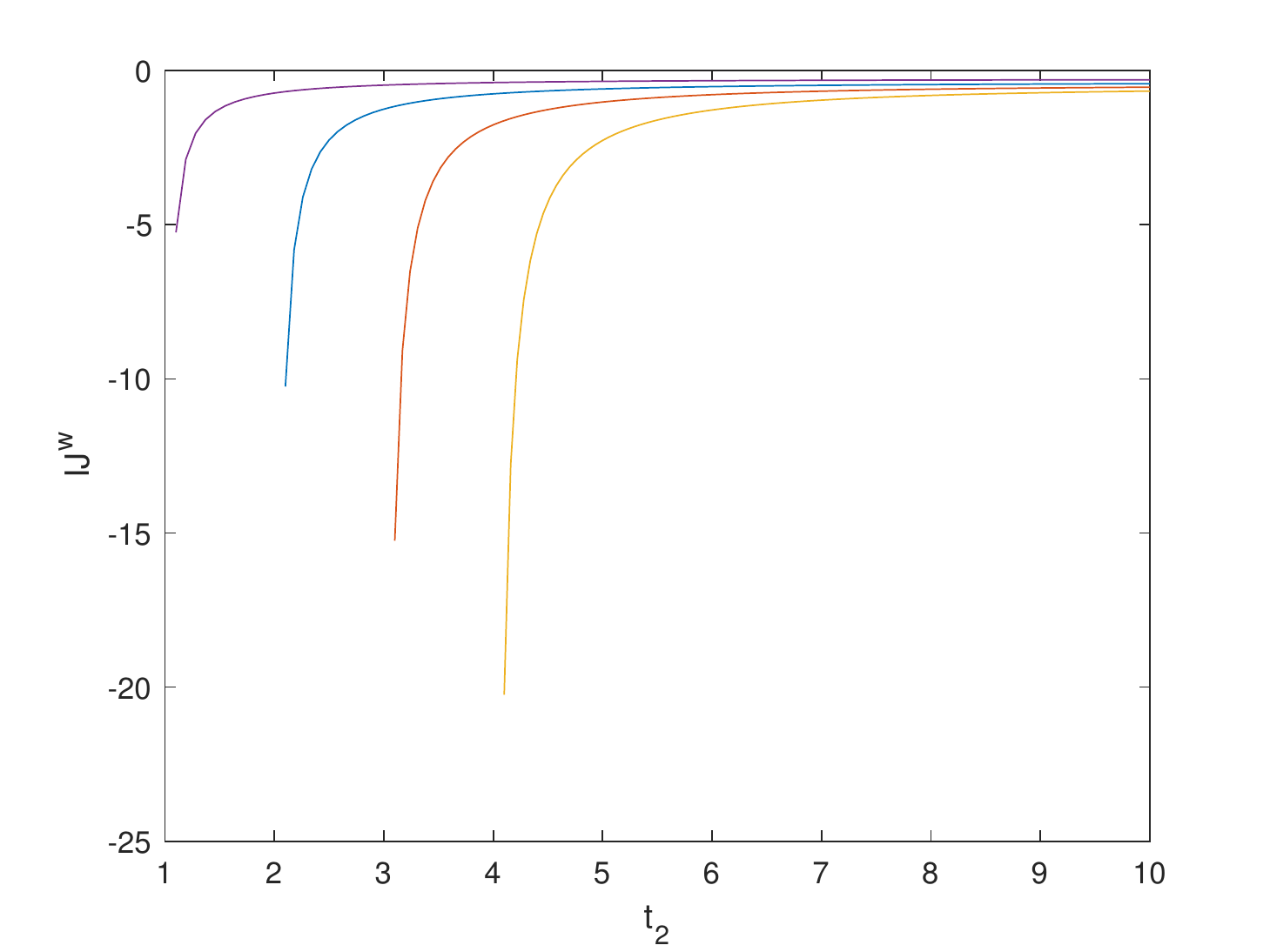}
         \caption{$IJ^w(t_1,t_2)$ as a function of $t_2$}
     \end{subfigure}
     \hfill
        \caption{Plot of $IJ^w$ in Example \ref{ex6} as a function of $t_1$ or $t_2$ fixing the other one with  $t_i=1$ (violet), 2 (blue), 3 (red) and 4 (yellow), $i=1,2$.}
        \label{fig-lognw}
\end{figure}
\end{example}

In the following theorem, we prove that the expression of the weighted extropy is composed of three terms in function of the weighted past, residual and interval extropies.

\begin{theorem}
Let $X$ be a random variable denoting the lifetime of a component. For all $0<t_1<t_2<+\infty$ the weighted extropy can be decomposed as follows:
\begin{eqnarray}
\nonumber
J^w(t_1,t_2)=F^2(t_1) J^w(_{t_1}X)+({F}(t_2)-{F}(t_1))^2 IJ^w(t_1,t_2)+ \overline{F}^2(t_2) J^w(X_t), \label{decomposition of WE}
\end{eqnarray}
i.e., Weighted Extropy is a function of WPEx, WREx and WIEx.
\end{theorem}

\begin{proof}
The proof is similar to the one of Theorem \ref{thm1} and hence it is omitted.
\end{proof}

In the following proposition, we analyze the effect of a linear transformation on the weighted interval extropy.

\begin{proposition}
\label{pr2}
Let $X$ be a non negative and absolutely continuous random variable and let $Y=aX+b$ where $a>0$ and $b\ge0$. The weighted interval extropy of $Y$ is given in terms of the interval extropy and weighted interval extropy of $X$ as
\begin{equation}
\label{linearwei}
IJ_Y^w(t_1,t_2)=IJ_X^w\left(\frac{t_1-b}{a},\frac{t_2-b}{a}\right)+\frac{b}{a}IJ_X\left(\frac{t_1-b}{a},\frac{t_2-b}{a}\right),
\end{equation}
where $t_1,t_2\in S_Y$.
\end{proposition}

\begin{proof}
By using the expressions of the cdf and the pdf of $Y$ in terms of $F_X$ and $f_X$ obtained in (\ref{utu}), the weighted interval extropy of $Y$ can be expressed as
\begin{eqnarray}
\nonumber
IJ_Y^w(t_1,t_2)&=&-\frac{1}{2\left(F_X\left(\frac{t_2-b}{a}\right)-F_X\left(\frac{t_1-b}{a}\right)\right)^2}\int_{t_1}^{t_2} \frac{x}{a^2}f_X^2\left(\frac{x-b}{a}\right)\,dx \\
\nonumber
&=& -\frac{1}{2\left(F_X\left(\frac{t_2-b}{a}\right)-F_X\left(\frac{t_1-b}{a}\right)\right)^2}\int_{\frac{t_1-b}{a}}^{\frac{t_2-b}{a}} x f_X^2(x)\,dx \\
\nonumber
\nonumber
&& -\frac{1}{2\left(F_X\left(\frac{t_2-b}{a}\right)-F_X\left(\frac{t_1-b}{a}\right)\right)^2}\int_{\frac{t_1-b}{a}}^{\frac{t_2-b}{a}} \frac{b}{a}f_X^2(x)\,dx \\
\nonumber
&=&IJ_X^w\left(\frac{t_1-b}{a},\frac{t_2-b}{a}\right)+\frac{b}{a}IJ_X\left(\frac{t_1-b}{a},\frac{t_2-b}{a}\right),
\end{eqnarray}
which completes the proof.
\end{proof}

\begin{remark}
The results given in Propositions \ref{pr1} and \ref{pr2} about linear transformations could be generalized to monotonic transformations but, in these cases, we do not obtain a formula of interest, in the sense that the interval extropy and the weighted interval extropy of the transformed random variable are not expressed in terms of the ones of the original random variable.
\end{remark}

In the following theorem, we present an upper bound for the weighted interval extropy given in terms of the generalized failure rate function.

\begin{theorem}
For an absolutely continuous non-negative random variable $X$, if the WIEx is increasing in $t_2$, then we have 
\begin{eqnarray}
IJ^w(t_1,t_2)\leq -\frac{t_2h_2(t_1,t_2)}{4}. \label{bound of WIE}
\end{eqnarray}
\end{theorem}

\begin{proof}
The proof follows in analogy with the one of Theorem \ref{thmub} and hence it is omitted.
\end{proof}

\section{Conclusion}
In this paper dynamic versions of extropy for double truncated random variables have been presented. Several examples are given. The behavior under linear transformations of these new measures has been studied. Some bounds for them have been found in relation with the Generalized Failure Rate.

\bmhead{Acknowledgments}
Francesco Buono and Maria Longobardi are members of the research group GNAMPA of INdAM (Istituto Nazionale di Alta Matematica), are partially supported by MIUR - PRIN 2017, project ‘‘Stochastic Models for Complex Systems’’, no. 2017 JFFHSH. The present work was developed within the activities of the project 000009\_ALTRI\_CDA\_75\_2021\_FRA\_LINEA\_B\_SIMONELLI funded by ``Programma per il finanziamento della ricerca di Ateneo - Linea B" of the University of Naples Federico II.

\section*{Conflict of interest}
The authors declare that they have no conflict of interest.


\begin{thebibliography}{}
%
%
\bibitem{barlow}
Barlow, R.~E., and F.J. Proschan (1996). Mathematical Theory of Reliability. Philadelphia: Society for Industrial and Applied Mathematics. 

\bibitem{Balakrishnan} Balakrishnan, N., Buono, F.,  Longobardi, M. (2020). On weighted extropies. Communications in Statistics-Theory and Methods, DOI:
10.1080/03610926.2020.1860222.

\bibitem{balakrishnan2} Balakrishnan, N., Buono, F.,  Longobardi, M. (2022). On Tsallis extropy with an application to pattern recognition. Statistics and Probability Letters, https://doi.org/10.1016/j.spl.2021.109241.

\bibitem{becerra}
Becerra, A., de la Rosa, J.I., González, E. et al. Training deep neural networks with non-uniform frame-level cost function for automatic speech recognition. Multimed Tools Appl 77, 27231–27267 (2018). https://doi.org/10.1007/s11042-018-5917-5.

\bibitem{Betensky}
Betensky, R. A. and Martin, E. C. (2003). Commentary: Failure-rate functions for doubly truncated random variables, IEEE Transactions on Reliability, 52(1), 7--8.

\bibitem{Di Crescenzo} Di Crescenzo, A., Longobardi, M. (2002). Entropy-based measure of uncertainty in past lifetime distributions. J. Appl. Probab. 39: 434--440.

\bibitem{Di Crescenzo2006} Di Crescenzo, A., Longobardi, M. (2006). On weighted residual and past entropies. Scientiae Mathematicae Japonicae, 64 (2), 255--266.

\bibitem{Ebrahimi} Ebrahimi, N (1996). How to measure uncertainty about residual lifetime. Sankhya A, 58, 48--57.

\bibitem{Kamari} Kamari, O.,  Buono, F. (2020). On extropy of past lifetime distribution. Ricerche di Matematica, DOI:10.1007/s11587-020-00488-7.

\bibitem{kazemi}
Kazemi, M.R., Tahmasebi, S., Buono, F., Longobardi, M. Fractional Deng Entropy and Extropy and Some Applications. Entropy 2021, 23, 623.
https://doi.org/10.3390/e23050623.

\bibitem{Khorashadizadeh}
Khorashadizadeh, M.; Rezaei Roknabadi, A. H. and Mohtashami Borzadaran, G. R. (2012). Characterizations of lifetime distributions based on doubly truncated mean residual life and mean past to failure, Communications in Statistics - Theory and Methods, 41(6), 1105--1115.


\bibitem{Krishnan} Krishnan, A. S., Sunoj, S. M.,  Nair, N. U. (2020). Some reliability properties of extropy for residual and past lifetime random variables. Journal of the Korean Statistical Society, 49(2), 457--474.

\bibitem{Lad} Lad, F., Sanfilippo, G.,  Agrò, G. (2015). Extropy: complementary dual of entropy. Statistical Science, 30(1), 40--58.

\bibitem{martinas}
Martinas, K., Frankowicz, M. (2000). Extropy-reformulation of the entropy principle. Periodica Polytechnica Chemical Engineering, 44(1): 29--38.

\bibitem{Misagh 2011} Misagh, F., Yari, G. H. (2011). On weighted interval entropy. Statistics and Probability Letters, 81(2), 188--194.

\bibitem{Misagh 2012} Misagh, F., Yari, G. (2012). Interval entropy and informative distance. Entropy, 14(3), 480-490.

\bibitem{Navarro} Navarro, J., Ruiz, J.M. (1996). Failure-rate function for doubly-truncated random variables. IEEE Transactions on Reliability, 4 , 685--690.

\bibitem{Poursaeed}
Poursaeed, M. H. and Nematollahi, A. R. (2008). On the mean past and the mean residual life under double monitoring, Communications in Statistics - Theory and Methods, 37(7), 1119--1133.

\bibitem{Qiu} Qiu, G., Jia, K. (2018). The residual extropy of order statistics. Stat. Probab. Lett. 133, 15--22.

\bibitem{Shannon} C. E. Shannon, (1948). A mathematical theory of communication. Bell System Technical J , 27, 379--423.

\bibitem{Sunoj} S. M. Sunoj, P. G. Sankaran, S. S. Maya (2009) Characterizations of Life Distributions Using Conditional Expectations of Doubly (Interval) Truncated Random Variables, Communications in Statistics - Theory and Methods, 38:9, 1441--1452, DOI: 10.1080/03610920802455001.

\end{thebibliography}
\end{document}